\RequirePackage[hyphens]{url}

\documentclass{article}
\usepackage{amsmath,amssymb,amsthm}
\usepackage{fullpage}
\usepackage{graphicx,xcolor}
\usepackage{dsfont}
\usepackage{booktabs}
\usepackage[round,sort]{natbib}
\usepackage[sort,nocompress,noadjust]{cite}
\usepackage[T1]{fontenc}
\usepackage[normalem]{ulem}
\usepackage[boxed]{algorithm2e}
\usepackage{hyperref}
\hypersetup{colorlinks,
            linkcolor=blue,
            citecolor=blue}
\usepackage{authblk}
\newtheorem{theorem}{Theorem}
\newtheorem{corollary}{Corollary}
\newtheorem{proposition}{Proposition}
\newtheorem{lemma}{Lemma}
\theoremstyle{definition}
\newtheorem{definition}{Definition}

\newcommand{\cB}{\mathcal{B}}

\newcommand{\cO}{\mathcal{O}}
\newcommand{\cP}{\mathcal{P}}

\newcommand{\cS}{\mathcal{S}}

\newcommand{\RR}{\mathbb{R}}

\newcommand{\ZZ}{\mathbb{Z}}

\newcommand{\bone}{\mathbf{1}}

\newcommand*{\ep}{\varepsilon}

\DeclareMathOperator*{\argmin}{argmin}

\DeclareMathOperator*{\conv}{conv}

\newcommand*{\ent}[1]{H(#1)}
\begin{document}

\title{An explicit analysis of the entropic penalty in linear programming}
\author{Jonathan Weed\thanks{Email:~\texttt{jweed@mit.edu}}\\Department of Mathematics, Massachusetts Institute of Technology}
\date{\today}
\maketitle
\begin{abstract}
Solving linear programs by using entropic penalization has recently attracted new interest in the optimization community, since this strategy forms the basis for the fastest-known algorithms for the optimal transport problem, with many applications in modern large-scale machine learning.
Crucial to these applications has been an analysis of how quickly solutions to the penalized program approach true optima to the original linear program.
More than 20 years ago, Cominetti and San Mart\'in showed that this convergence is exponentially fast; however, their proof is asymptotic and does not give any indication of how accurately the entropic program approximates the original program for any particular choice of the penalization parameter.
We close this long-standing gap in the literature regarding entropic penalization by giving a new proof of the exponential convergence, valid for any linear program.
Our proof is non-asymptotic, yields explicit constants, and has the virtue of being extremely simple.
We provide matching lower bounds and show that the entropic approach does not lead to a near-linear time approximation scheme for the linear assignment problem.
\end{abstract}

\section{Introduction}
In 1992, \citeauthor{Fan92} initiated the study of the \emph{entropic penalty} for linear programs.
Given a basic linear program of the form
\begin{equation}\tag{LP}\label{eq:primal}
\begin{array}{rl}
\min & c^\top x  \\
\text{subject to}& Ax  = b \\
& x  \geq 0\,,
\end{array}
\end{equation}
he proposed to solve instead the penalized program
\begin{equation}\tag{Pen}\label{eq:penalized}
\begin{array}{rl}
\min & c^\top x  - \eta^{-1} \ent{x}\\
\text{subject to}& Ax  = b\,,
\end{array}
\end{equation}
where $\ent{x} := \sum_i x_i \log \frac{1}{x_i}$ is the \emph{Shannon entropy} of $x$ viewed as a probability vector and $\eta > 0$ is a penalization parameter.
The term $- \eta^{-1} \ent{x}$ plays the role of a strongly convex regularizer, which also enforces the constraint $x \geq 0$.
As $\eta \to \infty$, we recover~\eqref{eq:primal}; however, one hopes that solving~\eqref{eq:penalized} is significantly easier.

Solving linear programs via entropic penalization does not initially seem like an especially attractive choice.
Unlike the well known logarithmic penalty, the entropic penalty is not self-concordant, which makes it a poor barrier function for interior point methods~\citep{BoyVan04}.
Nevertheless, the entropic penalty has been applied to various linear and nonlinear problems with empirical success~\citep{FanRajTsa97}.
It is also notable for its connection to other fields.
In statistics, it has been used as a tool for model selection and aggregation~\citep{JudRigTsy08,RigTsy11} and is intimately related to the maximum entropy principle for statistical inference~\citep{Jay82} and to maximum likelihood estimation~\citep{ChrHer00}.
The entropic penalty is also closely connected to first-order optimization methods such as mirror descent~\citep{Bub15} and to online learning algorithms for combinatorially structured problems~\citep{FreSch97,CesLug06,HelWar09,KooWarKiv10,AudBubLug13}.

The recent resurgence of interest in the entropic penalty in the machine-learning community has been driven by the fact that it can be used to obtain state-of-the-art methods for the \emph{optimal transport problem}~\citep{Cut13,CutDou14,SolDe-Pey15,GenCutPey16,BenCarNen16,AltWeeRig17}.
The use of the entropic penalty for such problems dates back to \citet{Sch31}~\citep[see][]{Leo14} and to \citet{Bre67}, who noted the connection between the entropic penalty and the computation of a projection onto the feasible set $\{x: Ax = b, x \geq 0\}$ with respect to the generalized Kullbeck-Leibler divergence.

What makes this penalty especially useful for transport problems is that the solution to~\eqref{eq:penalized} can be computed quickly by a simple iterative algorithm known as the \emph{Sinkhorn} or \emph{RAS} algorithm~\citep{Sin67}.
This fact was popularized by~\citet{Cut13}, and his work led to widespread adoption of the entropic penalty for computing optimal transport.
The introduction of the entropic penalty makes an enormous difference in practice:
since optimal transport can be formulated as a linear program, it can, of course, be solved in polynomial time, but the numerical experiments conducted by~\citet{Cut13} indicate that the same linear program with an entropic penalty can be solved up to 10,000 times faster, as long as $\eta$ is not too large.
On the other hand, those experiments also showed that solving the penalized program becomes costly as $\eta$ increases.

This same phenomenon is present in theory as well as in practice.
A recent theoretical analysis~\citep[see][]{AltWeeRig17} of the method of~\citet{Cut13} suggests that the time required to solve the optimal transport problem via entropic regularization scales linearly with $\eta$.
Even the guarantees of the most recent algorithms for solving~\eqref{eq:penalized} decay as $\eta$ grows.
For example, when $\{Ax = b, x \geq 0\}$ defines the Birkhoff polytope~\citep{Bru06} of $n \times n$ doubly stochastic matrices, an approximate solution to~\eqref{eq:penalized} can be found in time%
\footnote{The notation $\tilde O(\cdot)$ hides polylogarithmic factors.}
$\tilde O(n^2 \|c\|_\infty \eta)$~\citep{CohMadTsi17,AllLiOli17}.
In particular, when $\eta = \tilde O(1)$, the algorithms of~\citet{CohMadTsi17} and~\citet{AllLiOli17} run in time which is \emph{nearly linear} in the size of the input.

Nevertheless, if we wish to obtain a good approximation to the solution of~\eqref{eq:primal}, $\eta$ cannot be taken too small.
In the $\eta \to 0$ limit, the solution to~\eqref{eq:penalized} converges to the maximum-entropy point in the feasible set, which may be far from the optimum.
If the goal is to approximately solve~\eqref{eq:primal}, then $\eta$ must be large enough that the solution of~\eqref{eq:penalized} is still close to an optimum of~\eqref{eq:primal}.

To summarize: the choice of $\eta$ is essential.
Too large, and the computational benefits of using the regularizer disappear;
too small, and the entropic term induces significant bias and~\eqref{eq:penalized} is a poor approximation to the original problem.
The chief aim of this work is to quantify this trade-off.

\subsection{Prior work}
The question of how well~\eqref{eq:penalized} approximates~\eqref{eq:primal} as a function of $\eta$ was studied by \citet{ComSan94}.%
\footnote{In fact, the main object of study of~\citet{ComSan94} is the program $\max_y \left\{b^\top y - \eta^{-1} \sum_i e^{- \eta (c_i - (A^\top y)_i)}\right\}$,
which is the dual of~\eqref{eq:penalized} with the entropy replaced by the similar function $\bar{H}(x) := \sum_{i} x_i \left(\log \frac{1}{x_i} + 1\right)$.
They refer to the penalty appearing in this dual program as the \emph{exponential penalty}.
Our analysis applies equally well to their setting, but we focus on the vanilla entropic penalty for clarity.}
They showed that, under mild conditions, the optimal solution $x^\eta$ to~\eqref{eq:penalized} approaches an optimal solution $x^*$ to~\eqref{eq:primal} exponentially fast, in the sense that
\begin{equation}\label{eq:linear_rate}
\lim_{\eta \to \infty} \frac{\|x^\eta - x^*\|}{\exp(- M \eta)} = 0
\end{equation}
for some $M > 0$.
However, their proof does not make it easy to determine the order of magnitude of $M$---and, in particular, its dependence on problem-specific quantities such as the dimension and size of the feasible set $\{x : Ax = b, x \geq 0\}$.
This result is nevertheless tantalizing, insofar as it suggests that $x^\eta$ will be close to $x^*$ even for relatively small values of $\eta$.
Of course, knowing the size of $M$ is crucial to making this idea precise.
Prior to this work, theirs was the most general analysis of~\eqref{eq:penalized} available.

\subsection{Our contribution}
In this work, motivated by the recent popularity of entropic penalization for optimal transport, we prove a version of~\eqref{eq:linear_rate} with easy-to-understand constants. Our analysis applies to any linear program of the form~\eqref{eq:primal}.
We show (Section~\ref{sec:upper}) that the quality of the penalized solution $x^\eta$ satisfies
\begin{equation*}
c^\top x^\eta - \min_{x \in \cP} c^\top x \leq \Delta \exp\left(- \eta \frac{\Delta}{R_1} + \frac{R_1 + R_H}{R_1}\right)\,,
\end{equation*}
where $\Delta$ is the gap in objective value between an optimal vertex and any suboptimal vertex, and $R_1$ and $R_H$ are the radius of the feasible set with respect to the $\ell_1$ norm and the entropy, respectively.
As a corollary, we obtain that the result~\eqref{eq:linear_rate} obtained by~\citet{ComSan94} holds for any $M < \Delta/R_1$.
In addition to making explicit their result, our proof has the virtue of being very simple, requiring only elementary facts about entropy.
Moreover, we show (Section~\ref{sec:lower}) that no general improvement in the dependence on $\Delta$, $R_1$, or $R_H$ is possible, even for the simplest possible example, where the feasible set is the probability simplex.

Finally, specializing to the Birkhoff polytope (Section~\ref{sec:assignment}), we obtain nearly matching upper and lower bounds on the quality of the solution as a function of $\eta$.
In particular, these imply that $\eta$ cannot be taken to be $\tilde O(1)$, so that the entropic penalty is not a magic bullet for the assignment problem.

\subsection{Assumptions}\label{sec:assumptions}
We assume throughout that $\cP : = \{x: A x = b, x \geq 0\}$ is bounded.
To ensure that~\eqref{eq:primal} is nontrivial, we assume that $\cP$ is nonempty and that $c^\top x$ is not constant over $\cP$.

\subsection{Quantities of interest}
For convenience, we collect here definitions of the three quantities $\Delta$, $R_1$, and $R_H$ appearing in our bounds.

\begin{definition}
Let $V$ be the set of vertices of $\cP$. The \emph{suboptimality gap} $\Delta$ is
\begin{equation*}
\min_{v \in \cS} c^\top (v - v^*)\,,
\end{equation*}
where $v^* \in \argmin_{v \in V} c^\top v$ and $\cS : = \{v \in V: c^\top v > c^\top v^*$\}.
\end{definition}

\begin{definition}
The \emph{$\ell_1$ radius $R_1$} of $\cP$ is $\max_{x \in \cP} \|x\|_1$.
\end{definition}

\begin{definition}
The \emph{entropic radius $R_H$} of $\cP$ is $\max_{x, y \in \cP} \ent{x} - \ent{y}$.
\end{definition}

\section{Upper bound}\label{sec:upper}
In this section, we prove our main bound on the quality of the solution of~\eqref{eq:penalized}.
Our proof recovers the result of~\citet{ComSan94} that the penalized solution approaches an optimal solution exponentially fast.
Before doing so, however, we first prove a much simpler and weaker bound, which we call the \emph{slow rate}~\citep[see][where this analysis also appeared]{FanTsa93,AltWeeRig17}:
\begin{proposition}[Slow rate]\label{prop:slow_rate}
For all $\eta > 0$,
\begin{equation*}
c^\top x^\eta - \min_{x \in \cP} c^\top x \leq \eta^{-1} (\ent{x^\eta} - \ent{x^*}) \leq \eta^{-1} R_H\,.
\end{equation*}
\end{proposition}
\begin{proof}
Denote by $x^*$ an optimal solution to~\eqref{eq:primal}.
Since $x^*$ and $x^\eta$ both lie in $\cP$ and $x^\eta$ is an optimal solution to~\eqref{eq:penalized},
\begin{equation*}
c^\top x^\eta - \eta^{-1} \ent{x^\eta} \leq c^\top x^* - \eta^{-1} \ent{x^*}\,.
\end{equation*}
The claim follows.
\end{proof}

Note that the slow rate is much worse than the fast rate we hope to prove; however, \citet{Rig17} noted that the slow rate is actually tight for an infinite-dimensional analogue of~\eqref{eq:primal}.
This indicates that the reason that a fast rate obtains for~\eqref{eq:primal} is that the finite-dimensional problem exhibits a \emph{suboptimality gap} \citep[known as an energy gap in the statistical physics literature; see][]{MezMon09}.
Intuitively, the slow rate dominates the convergence until $\eta$ is large enough that $x^\eta$ is concentrated near enough to the optimal solution; after this point, convergence occurs exponentially fast.
We will return to this point in Section~\ref{sec:lower}.

We now turn to the main result.
\begin{theorem}[Fast rate]\label{thm:main}
If $\eta \geq \frac{R_1 + R_H}{\Delta}$, then the optimal solution $x^\eta$ of~\eqref{eq:penalized} satisfies
\begin{equation*}
c^\top x^\eta - \min_{x \in \cP} c^\top x \leq \Delta \exp\Big(-\eta \frac{\Delta}{R_1} + \frac{R_1 + R_H}{R_1}\Big)\,.
\end{equation*}
\end{theorem}
Theorem~\ref{thm:main} implies a bound on the size of $\eta$ required to obtain a solution of desired accuracy: to obtain a solution $x^\eta$ satisfying $c^\top x^\eta - \min_{x \in \cP} c^\top x \leq \ep$, it suffices to take $\eta = \left(\frac{R_1}{\Delta}\log \frac{\Delta}{\ep}\right)_+ + \frac{R_1 + R_H}{\Delta}$, where $(x)_+ := \max\{x, 0\}$.

Note that Theorem~\ref{thm:main} only holds for $\eta$ sufficiently large.
The requirement that $\eta \geq \frac{R_1 + R_H}{\Delta}$ corresponds exactly to the requirement that the exponent appearing on the right side of the above equation is nonpositive.
In Section~\ref{sec:lower}, we show that this restriction is necessary, in the sense that there are penalized linear programs for which $x^\eta$ does not make appreciable progress towards the minimizer until $\eta = \Omega(\frac{R_1 + R_H}{\Delta})$.

The proof of Theorem~\ref{thm:main} is elementary and relies on three simple lemmas about the entropy function, which we now state.
These lemmas are easy to verify; proofs appear in Section~\ref{sec:lemmas}.
Recall the definition of the binary entropy function:
\begin{equation*}
h(\lambda) = \lambda \log \frac 1 \lambda + (1-\lambda) \log \frac{1}{1-\lambda} \quad \quad \forall \lambda \in [0, 1]\,.
\end{equation*}

\begin{lemma}\label{lem:weak_convexity}
If $x$ and $y$ are nonnegative vectors and $\lambda \in [0, 1]$, then
\begin{equation*}
\ent{\lambda x + (1-\lambda) y} \leq \lambda \ent{x} + (1-\lambda) \ent{y} + \max\{\|x\|_1, \|y\|_1\} \cdot h(\lambda)\,.
\end{equation*}
\end{lemma}

\begin{lemma}\label{lem:monotone}
The function $f(\lambda) := \alpha h(\lambda) + \beta \lambda$ is increasing on the interval $[0, \frac{\beta}{\alpha + \beta}]$.
\end{lemma}

\begin{lemma}\label{lem:binary_entropy_bound}
If $0 \leq \rho \leq 1$, then
\begin{equation*}
\frac{h(\rho)}{\rho} \leq \log \frac 1 \rho + 1\,.
\end{equation*}
\end{lemma}

\begin{proof}[Proof of Theorem~\ref{thm:main}]
Let $V$ be the vertices of $\cP$.
Write $\cO := \{v^* \in V: c^\top v^* = \min_{v \in V} c^\top v\}$ for the set of optimal vertex solutions for~\eqref{eq:primal}, and let $\cS := V \setminus \cO$ be the set of suboptimal vertices.
Since $x^\eta \in \cP = \conv(V)$, we can write
\begin{equation*}
x^\eta = \sum_{v \in \cO} \lambda_v v + \sum_{w \in \cS} \lambda_w w
\end{equation*}
for some nonnegative vector $\lambda$ satisfying $\sum_{v \in V} \lambda_v = 1$.
If we let $\gamma = \sum_{w \in \cS} \lambda_w$, then $x^\eta = (1-\gamma) x^* + \gamma \tilde x$, where $x^* \in \conv(\cO)$ and $\tilde x \in \conv(\cS)$.
Since $x^*$ is a convex combination of elements of $\cO$, it lies on the optimal face of $\cP$ and is an optimal solution to~\eqref{eq:primal}.
On the other hand, since $\tilde x$ is a convex combination of suboptimal vertices, $c^\top (\tilde x - x^*) \geq \Delta$.

Let $g(\eta) := c^\top (x^\eta -  x^*)$.
We first prove two simple bounds on this quantity.
First, we have a trivial lower bound:
\begin{equation}\label{eq:g_lower_bound}
g(\eta) = c^\top(x^\eta - x^*) = \gamma c^\top (\tilde x - x^*) \geq \Delta \gamma\,.
\end{equation}
On the other hand, Proposition~\ref{prop:slow_rate} implies
\begin{equation}\label{eq:g_upper_bound}
g(\eta) \leq \eta^{-1}(\ent{x^\eta} - \ent{x^*})\,.
\end{equation}

By Lemma~\ref{lem:weak_convexity},
\begin{align}
\ent{x^\eta} - \ent{x^*} & \leq \left((1-\gamma) \ent{x^*} + \gamma \ent{\tilde x} + R_1 h(\gamma)\right) - \ent{x^*} \nonumber \\
& = R_1 h(\gamma) + \gamma(\ent{\tilde x} - \ent{x^*}) \nonumber \\
& \leq R_1 h(\gamma) + R_H \gamma\,. \label{eq:gamma_bound}
\end{align}
By assumption, $\eta \geq \frac{R_1 + R_H}{\Delta}$, so using~\eqref{eq:g_lower_bound} and~\eqref{eq:g_upper_bound} yields
\begin{equation*}
\gamma \leq \frac{g(\eta)}{\Delta} \leq \frac{\eta^{-1} R_H}{\Delta} \leq \frac{R_H}{R_1 + R_H} < 1\,.
\end{equation*}
Lemma~\ref{lem:monotone} then implies that $R_1 h(\gamma) + R_H \gamma \leq R_1 h\left(\frac{g(\eta)}{\Delta}\right) + R_H\frac{g(\eta)}{\Delta}$.
We can combine this observation with~\eqref{eq:gamma_bound} to obtain
\begin{equation*}
\eta g(\eta) \leq \ent{x^\eta} - \ent{x^*} \leq R_1 h\Big(\frac{g(\eta)}{\Delta}\Big) + R_H \frac{g(\eta)}{\Delta}\,.
\end{equation*}
Writing $\rho = \frac{g(\eta)}{\Delta}$ yields the fixed-point equation
\begin{equation*}
(\eta \Delta - R_H) \rho \leq R_1 h(\rho)\,.
\end{equation*}
By Lemma~\ref{lem:binary_entropy_bound}, $\frac{h(\rho)}{\rho} \leq \log \frac 1 \rho + 1$ for all $\rho \leq 1$.
We obtain
\begin{equation*}
\frac{\eta \Delta - R_H}{R_1} - 1 \leq \log \frac 1 \rho\,,
\end{equation*}
and hence
\begin{equation*}
g(\eta) = \Delta \rho \leq \Delta \exp\Big(-\eta\frac{\Delta}{R_1} + \frac{R_1 + R_H}{R_1}\Big)\,,
\end{equation*}
as desired.
\end{proof}
We also obtain a corollary which establishes the distance of $x^\eta$ to the optimal face, which reproduces the result of~\citet{ComSan94}.
Let $F := \conv\{\cO\}$ be the optimal face of $\cP$ with respect to the objective $c^\top x$, and denote by $d_1(x, F)$ the $\ell_1$ distance of the point $x$ to $F$.
\begin{corollary}
If $\eta \geq \frac{R_1 + R_H}{\Delta}$, then
\begin{equation*}
d_1(x^\eta, F) \leq 2 R_1 \exp\Big(-\eta\frac{\Delta}{R_1} + \frac{R_1 + R_H}{R_1}\Big)\,.
\end{equation*}
In particular,
\begin{equation*}
\lim_{\eta \to \infty} \frac{d_1(x, F)}{\exp(-M \eta)} = 0
\end{equation*}
for any $M < \Delta/R_1$.
\end{corollary}
\begin{proof}
Using the notation of Theorem~\ref{thm:main}, we have that there exist points $x^*, \tilde x \in \cP$ such that $x^*$ is optimal and
\begin{equation*}
x^\eta = (1-\gamma) x^* + \gamma \tilde x\,,
\end{equation*}
for $\gamma \leq \frac{g(\eta)}{\Delta} \leq \exp\left(-\eta\frac{\Delta}{R_1} + \frac{R_1 + R_H}{R_1}\right)$.
We obtain
\begin{equation*}
d_1(x^\eta, F) \leq \|x^\eta - x^*\|_1 \leq \gamma \|\tilde x - x^*\|_1 \leq 2 \gamma R_1\,,
\end{equation*}
and the claim follows.
\end{proof}

The quantity $\Delta$ is quite brittle, since it can be affected by the presence of even a single almost-optimal vertex whose objective value is very close to that of the optimal vertex.
However, the definition of $\Delta$ can be relaxed slightly to account for this case, as the following corollary shows.
\begin{corollary}
For any $\tau > 0$, let $\cO_\tau:=\{v \in V: c^\top v - \min_{x \in \cP} c^\top x \leq \tau\}$ and $\Delta_\tau := \min_{\substack{w \in \cO_\tau \\ v \in V \setminus \cO_\tau}} c^\top (v - w)$.
If $\eta \geq \frac{R_1 + R_H}{\Delta_\tau}$, then the optimal solution $x^\eta$ of~\eqref{eq:penalized} satisfies
\begin{equation*}
c^\top x^\eta - \min_{x \in \cP} c^\top x \leq \Delta_\tau \exp\Big(-\eta \frac{\Delta_\tau}{R_1} + \frac{R_1 + R_H}{R_1}\Big) + \tau\,.
\end{equation*}
\end{corollary}
\begin{proof}
If $c^\top x^\eta - \min_{x \in \cP} c^\top x \leq \tau$, the claim is vacuous, so assume that $c^\top x^\eta - \min_{x \in \cP} c^\top x > \tau$.
Let $\tau^* := \max_{v \in \cO_\tau} c^\top v - \min_{x \in \cP} c^\top x$, and note that $\tau^* \leq \tau$.
Given an optimal solution $x^*$ to~\eqref{eq:primal}, let $\cP_\tau := \cP \cap \{x: c^\top (x - x^*) \geq \tau^*\}$.
We have $\min_{x \in \cP_\tau} c^\top x - \min_{x  \in \cP} c^\top x = \tau^* \leq \tau$.
Moreover, the suboptimality gap of $\cP_\tau$ is $\Delta_\tau$.

By assumption, $x^\eta \in \cP_\tau$, and so $x^\eta$ is the solution to~\eqref{eq:penalized} over the smaller polytope $\cP_\tau$.
Applying Theorem~\ref{thm:main} to $\cP_\tau$ yields the claim.
\end{proof}
While the quantities $R_1$ and $R_H$ are easy to calculate, evaluating the suboptimality gap $\Delta$ is not easy in general.
Nevertheless, as we noted above, intuition from statistical physics implies that some dependence on $\Delta$ is necessary to obtain exponential convergence, a point which we substantiate in Section~\ref{sec:lower}.
We note the obvious fact that this dependence can be removed for integral polytopes, which are a core object of study in combinatorial optimization~\citep{Sch03}.
\begin{corollary}\label{cor:integrality}
If $\cP$ is integral and the entries of $c$ are integers, then
\begin{equation*}
c^\top x^\eta - \min_{x \in \cP} c^\top x \leq \exp\Big(- \frac{\eta}{R_1} + \frac{R_1 + R_H}{R_1}\Big)
\end{equation*}
for all $\eta \geq R_1 + R_H$.
\end{corollary}
\begin{proof}
By definition, the vertices of $\cP$ have integer coordinates, so for any vertex $v$, if $c$ is an integer vector then $c^\top v \in \ZZ$.
Therefore if $v^*$ is an optimal vertex and $c^\top v > c^\top v^*$, then $c^\top (v - v^*) \geq 1$, so $\Delta \geq 1$.
\end{proof}
\section{Lower bound}\label{sec:lower}
In this section, we present an explicit example of a simple family of linear programs for which our analysis is tight, up to constant factors.
This example evinces the two phenomena present in Theorem~\ref{thm:main}: the convergence of $c^\top x^\eta$ to the optimum is slow until $\eta$ is of order $\frac{R_1 + R_H}{\Delta}$, and once this threshold is reached convergence happens at precisely the speed indicated in the upper bound.
This example also validates the intuition presented above about the necessary dependence on the suboptimality gap: exponentially fast convergence is obtained only when $c^\top x^\eta - \min_{x \in \cP} c^\top x \leq \Delta$.

Fix positive constants $\alpha$ and $\beta$ and a dimension $d \geq 2$.
Let $c \in \RR^d$ be given by
\begin{equation*}
c_i = \left\{\begin{array}{ll}
0 & \text{ if $i = 0$,} \\
\alpha & \text{ otherwise,}
\end{array}\right.
\end{equation*}
and consider the linear program
\begin{equation}\label{eq:simplex}
\begin{array}{rl}
\min & c^\top x  \\
\text{subject to}& \sum_i x_i  = \beta \\
& x  \geq 0\,,
\end{array}
\end{equation}
Note that the polytope $\cP$ defined by the constraints of~\eqref{eq:simplex} is a rescaled version of the $d$-dimensional probability simplex.
We focus on the following penalized program:
\begin{equation}\label{eq:mw}
\begin{array}{rl}
\min & c^\top x - \eta^{-1} \ent{x} \\
\text{subject to}& \sum_i x_i  = \beta\,.
\end{array}
\end{equation}
We make the following simple observations about~\eqref{eq:simplex} and~\eqref{eq:mw}:
\begin{itemize}
\item The unique optimal solution to~\eqref{eq:simplex} is $x^* = e_1$, the first elementary basis vector, and $c^\top x^* = 0$.
\item The maximum value of $c^\top x$ over $\cP$ is $\alpha \beta$, achieved at any vertex other than $e_1$.
\item For this polytope, $\Delta = \alpha \beta$, $R_1 = \beta$, and $R_H = \beta \log d$.
\end{itemize}

The penalized program~\eqref{eq:mw} has an explicit solution, which is given by a rescaled version of the Gibbs distribution~\citep{MezMon09}.
\begin{proposition}\label{prop:mw_explicit}
The optimal solution $x^\eta$ to~\eqref{eq:mw} is given by
\begin{equation*}
x^\eta_i = \frac{\beta e^{- \eta c_i}}{\sum_j e^{-\eta c_j}}\,.
\end{equation*}
\end{proposition}

The guarantee of Theorem~\ref{thm:main} requires that $\eta \geq \frac{R_1 + R_H}{R_H} = \frac{1+\log d}{\alpha}$.
We now show that when $\eta$ is significantly smaller than this quantity, the solution to the penalized program is far from the true optimum.
Indeed, the following proposition establishes that we cannot even achieve a constant-factor improvement over the \emph{maximum} value of $c^\top x$ over $\cP$ until $\eta$ is of order $\frac{\log d}{\alpha}$.

\begin{proposition}\label{prop:no_progress}
For any $\ep > 0$, if $\eta \leq \frac{\log \ep d}{\alpha}$, then $c^\top \eta \geq (1-\ep) \alpha \beta$.
\end{proposition}
\begin{proof}
We prove the contrapositive.
Note that $c^\top x^\eta = \alpha (\beta - x^\eta_0)$, so if $c^\top x^\eta < (1-\ep) \alpha \beta$ then $x^\eta_0 > \ep \beta$.
By Proposition~\ref{prop:mw_explicit}, we can write explicitly
\begin{equation*}
x^\eta_0 = \frac{\beta}{\sum_{j} e^{-\eta c_j}} = \frac{\beta}{1 + (d-1) e^{- \eta \alpha}}\,.
\end{equation*}
If $x^\eta_0 > \ep \beta$, then $d e^{-\eta \alpha} < 1 + (d-1) e^{-\eta \alpha} < \frac{1}{\ep}$, so $\eta > \frac{\log \ep d}{\alpha}$, as claimed.
\end{proof}

The next proposition shows that the Theorem~\ref{thm:main} is tight up to a small constant factor.
\begin{proposition}\label{prop:asymptotic_rate}
If $\eta \geq \frac{1 + \log d}{\alpha} = \frac{R_1 + R_H}{\Delta}$, then
\begin{equation*}
c^\top x^\eta \geq \frac{1}{9} \alpha \beta e^{- \eta \alpha + 1 + \log d} = \frac{1}{9} \Delta \exp\Big(-\eta \frac{\Delta}{R_1} + \frac{R_1 + R_H}{R_1}\Big)\,.
\end{equation*}
\end{proposition}
\begin{proof}
If $\eta \geq \frac{\log d + 1}{\alpha}$, then $\sum_j e^{-\eta c_j} = 1 + (d-1) e^{-\eta \alpha} \leq 1+\frac{d-1}{3 d}$.
Using Proposition~\ref{prop:mw_explicit},
\begin{equation*}
c^\top x^\eta = (d-1) \alpha \frac{\beta e^{- \eta \alpha}}{\sum_j e^{-\eta c_j}} \geq \frac 1 3 d \alpha \beta e^{-\eta \alpha} \geq \frac{1}{9} \alpha \beta e^{-\eta \alpha + 1 + \log d}\,,
\end{equation*}
as claimed.
\end{proof}

\section{Entropic penalization for the assignment problem}\label{sec:assignment}
In this section, we given an application of Theorem~\ref{thm:main} to the \emph{assignment problem}, a fundamental combinatorial optimization problem~\citep{Sch03}.
Our motivation for analyzing this example explicitly is twofold.
First, this is a case where entropic penalization has already been proposed as a good candidate algorithm~\citep{KosYui94,ShaGauGri11}.
Second, as noted in the introduction, new fast algorithms for the \emph{matrix scaling problem}~\citep{CohMadTsi17,AllLiOli17} show that a penalized version of the assignment problem with cost matrix $C$ can be solved in time $\tilde O(n^2 \|C\|_\infty \eta)$.
These fast algorithms raise the prospect that entropic penalization could provide a near-linear time algorithm for the assignment problem, a major breakthrough~\citep[see, e.g.,][]{Mad13}.

Whether this breakthrough is possible depends crucially on the size of $\eta$ required to solve the problem accurately.
The best bounds available from previous works on the problem~\citep{KosYui94,ShaGauGri11} require $\eta \gtrsim n \log n$ to achieve constant accuracy, which is just the guarantee given by the slow rate (Proposition~\ref{prop:slow_rate}).
An open question implicit in these works is whether this is optimal, or whether $\eta = \tilde O(1)$ suffices.
(In particular, this would imply a near-linear time algorithm for the assignment problem.)
By applying Theorem~\ref{thm:main} and exhibiting an almost-matching lower bound, we show exactly what rates are attainable for the Birkhoff polytope.
In short, our hopes are dashed: $\eta$ cannot be taken to be dimension free in general.

We first recall the problem.
Given a bipartite graph with edge weights, the goal of the assignment problem is to find a minimum-cost perfect matching in the graph.
This problem also has a well known linear programing formulation: given a matrix $C \in \RR^{n \times n}$ of edge weights, the assignment problem is
\begin{equation}\label{eq:assignment}
\begin{array}{rl}
\min & \langle C, X \rangle  \\
\text{subject to}& X \bone = \bone \\
& X^\top \bone = \bone \\
& X  \geq 0\,,
\end{array}
\end{equation}
The polytope given by the constraints $\{X \bone = \bone, X^\top \bone = \bone, X \geq 0\}$ is known as the \emph{Birkhoff polytope}, and its vertices are the permutation matrices~\citep{Bru06}, a result known as the Birkhoff-von Neumann Theorem.

We first give an upper bound on the quality of $X^\eta$ as a function of the regularization parameter $\eta$.
We require a preliminary lemma, whose proof appears in Section~\ref{sec:lemmas}.

\begin{lemma}\label{lem:assignment_parameters}
The Birkhoff polytope $\cB$ has $R_1 = n$ and $R_H = n \log n$.
\end{lemma}

Lemma~\ref{lem:assignment_parameters} combined with Theorem~\ref{thm:main} yields the following guarantee for the entropic penalty applied to the assignment problem.
For normalization purposes, we assume that the entries of $C$ are nonnegative integers, as is common in the combinatorial optimization literature.
\begin{proposition}\label{prop:assignment_upper}
An additive $\ep$ approximation to the assignment problem with cost matrix $C \in \ZZ_{\geq 0}^{n \times n}$ can be found by solving an entropy-penalized version of~\eqref{eq:assignment} with parameter $\eta = O\left(n \log \frac n \ep\right)$.
\end{proposition}
\begin{proof}
Since the entries of $C$ are integers and the Birkhoff polytope is integral~\citep{Sch03}, Corollary~\ref{cor:integrality} implies $\Delta \geq 1$.
Theorem~\ref{thm:main} implies that as long as $\ep < 1$, we can obtain a solution $\hat X$ such that
\begin{equation*}
\langle C, \hat X \rangle - \min_{X^* \in \cB} \langle C, X^* \rangle \leq \ep
\end{equation*}
by solving the penalized program with $\eta = \frac{R_1}{\Delta}\log \frac \Delta \ep + \frac{R_1 + R_H}{\Delta} \leq n \log \frac 1 \ep + n(1 + \log n)$, which is $O\left(n \log \frac n \ep\right)$.
\end{proof}
Proposition~\ref{prop:assignment_upper} is disappointing: it guarantees exponential convergence of $X^\eta$ to $X^*$ only when $\eta \gtrsim n \log n$, a far cry from the hoped-for result that $\eta$ could be taken $\tilde O(1)$.
We now show that, up to logarithmic factors, this bound is tight.
The following theorem implies that, even when $C \in \{0, 1\}^{n \times n}$, $\langle C, X^\eta \rangle$ can be bounded away from the optimal value if $\eta \ll n$.

\begin{theorem}
Let $C \in \{0, 1\}^{n \times n}$ be the matrix given by
\begin{equation}\label{eq:cost_def}
C_{ij} := \left\{\begin{array}{ll}
0 & \text{ if $j = i$ or $j = i+1$,} \\
1 & \text{ otherwise.}
\end{array}\right.
\end{equation}
If $\eta \leq n \log \frac{1-\ep}{\ep}$,
then
\begin{equation*}
\langle C, X^\eta \rangle - \langle C, X^* \rangle \geq \ep
\end{equation*}
\end{theorem}
\begin{proof}
The matrix defied in~\eqref{eq:cost_def} admits the unique optimum solution $X^* = I$, the identity matrix, and optimal value $\langle C, I \rangle = 0$.
For any permutation $\Pi \neq I$, on the other hand, $\langle C, \Pi \rangle \geq 1$.

We prove the contrapositive.
By the Birkhoff-von Neumann theorem, we can write $X^\eta$ as a convex combination of permutation matrices:
\begin{equation*}
X^\eta = \lambda_I I + \sum_{\Pi \neq I} \lambda_\Pi \Pi\,.
\end{equation*}
By assumption, $\langle C, X^\eta \rangle < \ep$, so $\lambda_I > 1 - \ep$.
This implies that $X^\eta_{ii} > 1-\ep$ for $1 \leq i \leq n$, and therefore that $X^\eta_{i, i+1} < \ep$ for $1 \leq i \leq n-1$.

Sinkhorn's theorem~\citep{Sin67} combined with first-order optimality conditions for the penalized program guarantee that $X^\eta = D_1 A D_2$ for positive diagonal matrices $D_1$ and $D_2$ and $A_{ij} := \exp(- \eta C_{ij})$.
Write $d$ for the vector of diagonal entries of $D_2$.
For $1 \leq i \leq n-1$, we have
\begin{equation}\label{eq:i_leq_n-1}
\frac{d_i}{d_{i+1}} = \frac{A_{i,i}}{A_{i, i+1}}\frac{d_i}{d_{i+1}} = \frac{X^\eta_{i,i}}{X^\eta_{i, i+1}} > \frac{1-\ep}{\ep}\,.
\end{equation}

Finally, we note that
\begin{equation*}
\frac{X^\eta_{n, n}}{X^\eta_{n, 1}} = \frac{A_{n,n}}{A_{n, 1}}\frac{d_n}{d_{1}} = \exp(\eta) \frac{d_n}{d_{1}}\,,
\end{equation*}
and since $X^\eta_{n, 1} < \ep$, we obtain
\begin{equation}\label{eq:i_eq_n}
\frac{d_n}{d_1} > \exp(-\eta) \frac{1-\ep}{\ep}\,.
\end{equation}

Combining~\eqref{eq:i_leq_n-1} for $1 \leq i \leq n-1$ with~\eqref{eq:i_eq_n} yields
\begin{equation*}
0 > n \log \frac{1-\ep}{\ep} - \eta\,,
\end{equation*}
which implies
\begin{equation*}
\eta > n \log \frac{1-\ep}{\ep}\,,
\end{equation*}
as claimed.
\end{proof}

\section{Conclusions}
Our focus in this work has been on making explicit the asymptotic analysis of~\citet{ComSan94}.
Their paper has been cited consistently in the computational optimal transport community as giving the best account of the speed of convergence of the penalized program to the original linear program~\citep[see][]{GenCutPey16,BenCarCut15,BenCarNen16,BloSegRol17,CarDuvPey17,DenPetSch14,DesPapRou16,Di-GerNen17,DiaRauRiv15,GenCutPey16,Sch16,PeyCut17,LuiRudPon18}.
We hope that the simple and explicit proof here will clarify the nature of the exponential rate proved in~\citet{ComSan94}, and provide a framework for a more refined analysis of the entropic penalty for linear programs of interest.

One puzzle that remains is to give theoretical justification to the observation of~\citet{Cut13} that small values of $\eta$ achieve good accuracy on real-world optimal transport data.
It is clear that the analysis of Theorem~\ref{thm:main} could be improved via a more refined understanding of the ``energy spectrum'' of optimal transport (i.e., the size and structure of the set of nearly-optimal transports), but obtaining this understanding even in the case where the costs are i.i.d.\ random variables is a very deep question~\citep{Ald01}. We leave obtaining a more sophisticated grasp on the behavior of this trajectory for future work.

\section{Proofs of Lemmas}\label{sec:lemmas}
\begin{proof}[Proof of Lemma~\ref{lem:weak_convexity}]
Write $\bar \lambda = 1 - \lambda$, and let $R = \max\{\|x\|_1, \|y\|_1\}$.
\begin{align*}
\ent{\lambda x + \bar \lambda y} & = \sum_i (\lambda x_i + \bar \lambda y_i) \left(\log \frac{1}{\lambda x_i + \bar \lambda y_i}\right) \\
& = \sum_i \lambda x_i \left(\log \frac{1}{\lambda x_i + \bar \lambda y_i}\right) + \sum_i \bar \lambda y_i \left(\log \frac{1}{\lambda x_i + \bar \lambda y_i}\right) \\
& \leq \sum_i \lambda x_i \left(\log \frac{1}{\lambda x_i}\right) + \sum_i \bar \lambda y_i \left(\log \frac{1}{\bar \lambda y_i}\right) \\
& = \lambda \ent{x} + \bar \lambda \ent{y} + \lambda \log \frac 1 \lambda \sum_i x_i + \bar \lambda \log \frac{1}{\bar \lambda} \sum_i y_i \\
& \leq \lambda \ent{x} + \bar \lambda \ent{y} + Rh(\lambda)\,.
\end{align*}
\end{proof}
\begin{proof}[Proof of Lemma~\ref{lem:monotone}]
The derivative of $f$ satisfies
\begin{equation*}
f'(\lambda) = \alpha \log \frac{1-\lambda}{\lambda} + \beta\,.
\end{equation*}
When $\lambda \in (0, \frac{\beta}{\alpha + \beta}]$,
\begin{equation*}
\log \frac{1-\lambda}{\lambda} \geq - \log \frac{\beta}{\alpha} > - \frac{\beta}{\alpha}\,,
\end{equation*}
so $f'(\lambda) > 0$.
The claim follows.
\end{proof}

\begin{proof}[Proof of Lemma~\ref{lem:binary_entropy_bound}]
By definition $h(\rho) = \rho \log \frac 1 \rho + (1-\rho) \log \frac{1}{1-\rho}$, so it suffices to show that
\begin{equation*}
\frac{(1-\rho)}{\rho} \log \frac{1}{1-\rho} \leq 1 \quad \quad \forall \rho \in [0, 1]\,.
\end{equation*}
This inequality is easily verified by noting that the derivative of the left side is nonpositive on $(0, 1)$ and $\lim_{\rho \to 0^+} \frac{(1-\rho)}{\rho} \log \frac{1}{1-\rho} = 1$.
\end{proof}

\begin{proof}[Proof of Lemma~\ref{lem:assignment_parameters}]
It is trivial to see that all $X \in \cB$ satisfy $\sum_{ij} X_{ij} = n$, so $R_1 = n$.
For any $X \in \cB$,
\begin{align*}
\ent{X} & = \sum_{ij} X_{ij} \log \frac{1}{X_{ij}} \\
& = \sum_i \sum_j X_{ij} \log \frac{1}{X_{ij}} \\
& = \sum_i \ent{X_i}\,,
\end{align*}
where $X_i$ denotes the $i$th row of $X$.
Since each row of $X$ is a nonnegative vector of dimension $n$ whose entries sum to $1$, for each $1 \leq i \leq n$ the bound $0 \leq H(X_i) \leq \log n$ holds.
Therefore $0 \leq \ent{X} \leq n \log n$ for all $X \in \cB$, which proves that $R_H \leq n \log n$.
\end{proof}

\section{Acknowledgments}
This work was supported in part by NSF Graduate Research Fellowship DGE-1122374. The author would like to thank J.~Altschuler and P.~Rigollet for useful discussions, as well as the anonymous referees for their suggestions.

\bibliographystyle{plainnat}
\bibliography{Weed18}
\end{document}